\documentclass{llncs}
\usepackage{makeidx}
\usepackage{amsfonts}
\usepackage{amsmath}
\usepackage{amssymb}
\usepackage[dvips]{graphicx}
\usepackage{verbatim}

\newcommand{\heading}[1]{\medskip\par\noindent{\bf #1}}
\def\hom{\hbox{\textsc{Hom}}}
\def\cover{\hbox{\textsc{Cover}}}
\def\planarcover{\hbox{\textsc{PlanarCover}}}
\def\segcol{\hbox{\textsc{SegmentColoring}}}
\def\mpsat{\hbox{\textsc{2-in-4-MonotonePlanarSAT}}}
\def\cP{\hbox{\rm \sffamily P}}
\def\cNP{\hbox{\rm \sffamily NP}}

\def\computationproblem#1#2#3{
	\begin{center}
	\begin{tabular}{rp{10cm}}
	{\bf Problem:\enspace}&#1\\
	{\bf Input:\enspace}&#2\\
	{\bf Output:\enspace}&#3\\
	\end{tabular}
	\end{center}
}

\newenvironment{packed_enum}{
	\begin{enumerate}
		\setlength{\itemsep}{1pt}
	    \setlength{\parskip}{0pt}
		\setlength{\parsep}{0pt}
}{\end{enumerate}}

\newenvironment{packed_item}{
	\begin{itemize}
		\setlength{\itemsep}{1pt}
	    \setlength{\parskip}{0pt}
		\setlength{\parsep}{0pt}
}{\end{itemize}}

\begin{document}

\mainmatter

\title{On the Complexity of Planar Covering of Small Graphs\thanks{The initial research was
supported by DIMACS/DIMATIA REU program (grant number 0648985). The third and the fourth author
were supported by Charles University as GAUK 95710.  The fourth author is also affiliated
with Institute for Theoretical Computer Science (supported by project 1M0545 of The Ministry
of Education of the Czech Republic).
}}
\titlerunning{On the Complexity of Planar Covering of Small Graphs}
\newcounter{lth}
\setcounter{lth}{2}
\author{Ond\v{r}ej B\'ilka\thanks{
Department of Applied Mathematics, Faculty of Mathematics and Physics, Charles
University, Malostransk\'e n\'am.~25, 118~00 Prague, Czech Republic.
E-mails: {\tt ondra@kam.mff.cuni.cz, jirasekjozef@gmail.com, klavik@kam.mff.cuni.cz,
	tancer@kam.mff.cuni.cz, janv@kam.mff.cuni.cz}
}
 \and Jozef Jir\'asek$^{\fnsymbol{lth}}$ \and Pavel Klav\'ik$^{\fnsymbol{lth}}$ \and 
 Martin Tancer$^{\fnsymbol{lth}}$ \and Jan Volec$^{\fnsymbol{lth}}$
}
\authorrunning{O.~B\'ilka, J.~Jir\'asek, P.~Klav\'\i k, M.~Tancer, J.~Volec}
\institute{}
\maketitle

\begin{abstract}
The problem $\cover(H)$ asks whether an input graph $G$ covers a fixed graph
$H$ (i.e., whether there exists a homomorphism $G \rightarrow H$ which locally
preserves the structure of the graphs). Complexity of this problem has been
intensively studied. In this
paper, we consider the problem $\planarcover(H)$ which restricts the input
graph $G$ to be planar.

$\planarcover(H)$ is polynomially solvable if $\cover(H)$ belongs to \cP, and
it is even trivially solvable if $H$ has no planar cover. Thus the interesting
cases are when $H$ admits a planar cover, but $\cover(H)$ is \cNP-complete.
This also relates the problem to the long-standing Negami Conjecture which aims
to describe all graphs having a
planar cover. Kratochv\'il asked whether there are non-trivial graphs for which
$\cover(H)$ is \cNP-complete but $\planarcover(H)$ belongs to \cP.

We examine the first nontrivial cases of graphs $H$ for which $\cover(H)$ is
\cNP-complete and which admit a planar cover. We prove
\cNP-completeness of $\planarcover(H)$ in these cases.

\end{abstract}

\section{Introduction}

Unless stated otherwise, we work with simple undirected finite graphs and we use standard notation from graph theory.

\heading{Graph Homomorphisms and Covers.} Let $G$ and $H$ be graphs. A mapping $f : V(G) \to V(H)$
is a \emph{homomorphism} from $G$ to $H$ if the edges of $G$ are mapped to the edges $H$, i.e., for
every edge $uv\in E(G)$, $f(u)f(v) \in E(H)$.

A homomorphism $f$ is called \emph{locally bijective} if for every $v \in V(G)$
the \emph{closed neighborhood} $N_G[v] \subseteq V(G)$ is bijectively mapped to
$N_H[f(v)] \subseteq V(G)$. Notice that $x,y \in N_G[v]$ and $f(x)f(y) \in E(H)$ may or may not
imply $xy \in E(G)$. We say that $G$ \emph{covers} $H$ (or $G$ is a \emph{cover} of $H$) if there
exists a locally bijective homomorphism from $G$ to $H$; see Figure~\ref{example_of_covering}. If
$G$ covers $H$, their local structures are somewhat similar. Note that if $G$ covers $H$ and $H$ is
connected, then $|V(G)|$ is a multiple of $|V(H)|$ and every vertex of $H$ has the same number of
preimages.

\begin{figure}[t]
\centering
\includegraphics{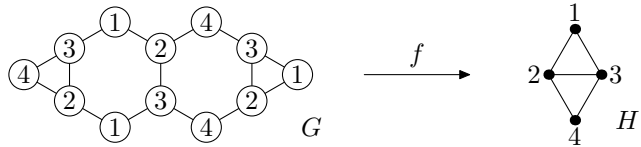}
\caption{An example of a cover of $K_4^-$. For every vertex $v$, its image $f(v)$ is written in the
	circle. Notice that for vertices mapped to $1$, its neighbors $2$ and $3$ may or may not be
	adjacent.}
\label{example_of_covering}
\end{figure}

Graph homomorphisms and covers provide a common language for various problems in graph theory. They
are studied as generalized coloring. A graph is properly $k$-colorable if and only if it
homomorphically maps to $K_k$. Similarly, covering is a generalization of coloring which puts
additional restrictions on neighborhoods. From another point of view, $G$ covers $K_k$ if and only if $G$
can be partitioned into $k$ 1-perfect codes. 

\heading{Computational Problems.} For a fixed graph $H$, the problem $\hom(H)$ asks whether there
exists a homomorphism from an input graph $G$ to $H$. Hell and
Ne\v{s}et\v{r}il~\cite{hell_nesetril90} proved a dichotomy for computational complexity: $\hom(H)$
is polynomially solvable if $H$ is a bipartite graph, and it is \cNP-complete otherwise.

Similarly, the problem $\cover(H)$ asks whether an input graph $G$ covers a fixed graph $H$. Study
of this problem was pioneered by Bodlaender~\cite{bodlaender}. First results depending on the graph
$H$ were proved by Abello et al.~\cite{abello}.  Kratochv\'{\i}l~\cite{kratochvil94} showed that
$\cover(K_4)$ is \cNP-complete. Afterwards, Kratochv\'{\i}l, Proskurowski and
Telle~\cite{kratochvil97} and Fiala~\cite{fiala00} proved that $\cover(H)$ is \cNP-complete for
every $k$-regular graph $H$ with $k \ge 3$.  Later, a dichotomy for all simple graphs with up to six
vertices was proved,~\cite{kratochvil98}.

We can restrict the input graph $G$ and ask whether it changes $\cover(H)$ to be polynomially
solvable. In this paper, we restrict $G$ to be planar.  We consider the following problem:

\computationproblem
{$\planarcover(H)$.}
{A planar graph $G$.}
{Yes if $G$ covers $H$, no otherwise.}

Every $\planarcover(H)$ problem trivially lies in \cNP. In the rest of the paper, we only question \cNP-hardness. Note that if $\cover(H)$ is
polynomially solvable, then $\planarcover(H)$ is also polynomially solvable.

Many \cNP-complete problems remain hard for planar inputs. Originally, the graph covering problems
looked similar to problems such as Not-All-Equal Satisfiability or 3-Edge-Colorability of Cubic
graphs. Both of them are polynomially solvable for planar graphs, the first was proved by
Moret~\cite{planar_naesat}, the latter is trivial to decide since Tait has showed that the
3-Edge-Colorability of Cubic planar graphs is equivalent to the Four Color Theorem. Indeed the
NP-hardness reduction for $\cover(K_4)$ presented in~\cite{kratochvil94} is from 3-Edge-Colorability
of Cubic graphs. This has led Kratochv\'\i l to pose the problem of the complexity of
$\planarcover(K_4)$ at several occasions, including 6th Czech-Slovak Symposium on Graph Theory 2006,
IWOCA 2007 and ATCAGC 2009.  In this paper, we prove that $\planarcover(K_4)$ is \cNP-complete.

\heading{Negami Conjecture.} As a motivation to study the $\planarcover$ problems, we describe a relation to a long-standing conjecture of
Negami~\cite{negami88}. For a graph $H$, we can ask whether there exists a planar graph $G$ that covers $H$. If the answer is no, then the
problem $\planarcover(H)$ is trivial---we output ``no'' regardless of the input. Negami conjectured the following:

\begin{conjecture}
The connected graphs which admit a planar cover are exactly the connected graphs embeddable in the projective plane.
\end{conjecture}

If a connected graph is embeddable in the projective plane, it is straightforward to construct one
of its planar covers. The other implication is still open. The most recent results can be found in
Hlin\v{e}n\'y and Thomas~\cite{hlineny04}. For example, $K_7$ has no (finite) planar cover since
there is no six regular planar graph.  Therefore, the problem $\planarcover(K_7)$ is polynomially
solvable but $\cover(K_7)$ is \cNP-complete; see~\cite{kratochvil97}.

It is natural to ask whether the restriction to planarity makes the problem easier in a non-trivial
case. Kratochv\'{\i}l asked the following question in his talk at Prague Midsummer Combinatorial
Workshop 2009:

\begin{question} \label{kratochvil_question}
Is it true that $\planarcover(H)$ is \cNP-complete if and only if $\cover(H)$ is \cNP-complete and the graph $H$ admits a planar cover?
\end{question}

\heading{Our Results.} In this paper, we show that $\planarcover(H)$ is \cNP-complete for several small graphs $H$.

\begin{theorem} \label{K6_is_NPc_theorem}
The problem $\planarcover(K_6)$ is \cNP-complete.
\end{theorem}

The graph $K_6$ is a somewhat extremal case for the Negami Conjecture. If a planar graph $G$ covers $K_6$, it has to be $5$-regular. The structure of
$5$-regular planar graphs is very limited, but we show that the problem is still \cNP-complete.

\begin{theorem} \label{K4_K5_are_NPc_theorem}
The problems $\planarcover(K_4)$ and $\planarcover(K_5)$ are \cNP-complete.
\end{theorem}

Covering of these regular graphs is related to coloring squares of graphs. For example, a cubic planar graph $G$ covers $K_4$ if and only if its
square $G^2$ is $4$-colorable. Coloring the squares of graphs (especially planar) is widely studied as a special case of the channel assignment
problem; see~\cite{ramanathan93}. Dvo\v{r}\'ak et al.~\cite[Theorem 25]{dvorak08} prove that deciding whether the square of a given subcubic planar
graph is 4-colorable is \cNP-complete. Theorem~\ref{K4_K5_are_NPc_theorem} strengthens this result.

\medskip

We denote $K_4$ with a leaf attached to a vertex by $K_4^+$ and $K_5$ without an edge by $K_5^-$; see Figure~\ref{graphs_K4plus_and_K5minus}. 

\begin{theorem} \label{K4plus_K5minus_are_NPc_theorem}
The problems $\planarcover(K_4^+)$ and $\planarcover(K_5^-)$ are \cNP-complete.
\end{theorem}

Theorems~\ref{K6_is_NPc_theorem}, \ref{K4_K5_are_NPc_theorem} and \ref{K4plus_K5minus_are_NPc_theorem} together give an affirmative answer
to Question~\ref{kratochvil_question} for all graphs with up to five vertices except for $W_4$; see details in Conclusions.

\medskip

We also examine the smallest non-trivial multigraph case. The dumbbell graph $D$ is a multigraph with two adjacent vertices with a loop on each vertex; see Figure~\ref{multigraph_example} on the right. This
graph is the smallest multigraph for which the problem $\cover$ is \cNP-complete.

\begin{theorem} \label{cinka_is_NPc}
The problem $\planarcover(D)$ is \cNP-complete.
\end{theorem}

This result strengthens a result of Janczewski et. al~\cite[Proposition 5]{polaci} which proves
hardness for partial $\planarcover(D)$.  By partial covers we mean locally \emph{injective}
homomorphisms. As described in Section~\ref{sec:hardness_dumbbell}, if $G$ covers $D$, it has to be
a cubic planar graph. For a partial cover of $D$, it has to be a subcubic planar graph. But if the input graph is
cubic, then every partial cover of $D$ is also a cover of $D$.\footnote{We note that even in general
partial $\planarcover(H)$ problem is at least as hard as $\planarcover(H)$.} On the other hand, reductions for
partial covers cannot be easily extended to covers while preserving planarity. 


\section{Hardness of Planar Covering of $K_6$} \label{sec:hardness_K6}

In this section, we prove Theorem~\ref{K6_is_NPc_theorem}: $\planarcover(K_6)$ is \cNP-complete.
First, we describe a problem we reduce from.

An \emph{intersection representation} of a graph is an assignment of sets to the vertices in such a
way that two vertices are adjacent if and only if the corresponding sets intersect. A graph is
called a \emph{segment graph} if it has an intersection representations where the sets are segments
in the plane. We consider only segment representations with all endpoints distinct and with no three
segments crossing in one point.  

\computationproblem
{$k$-\segcol}
{A segment representation of a graph $G$.}
{Yes if $G$ is $k$-colorable, no otherwise.}

Ehrlich et al.~\cite{tarjan76} proved that $k$-\segcol\ is \cNP-complete for $k \ge 3$. We note that there exist segment graphs which have every
representation exponentially large in the number of vertices;
see~\cite{kratmat_seg}. However this representation is a part of the input
hence it does not pose a problem; see~\cite{tarjan76}.

\heading{Overview of the Reduction.} We reduce $\planarcover(K_6)$ from $6$-\segcol. For a graph $G$ with a segment representation, we construct a
plane graph $G'$ which covers $K_6$ if and only if $G$ is $6$-colorable.

The reduction is sketched in Figure~\ref{reduction_sketch}. Consider an arrangement of segments. Every segment is split by crossings into several
\emph{subsegments} which contain no crossings. We construct a graph $G'$ with the same topology as the segment representation of $G$. Every subsegment
is represented by two parallel edges. 

\begin{figure}[t]
\centering
\includegraphics{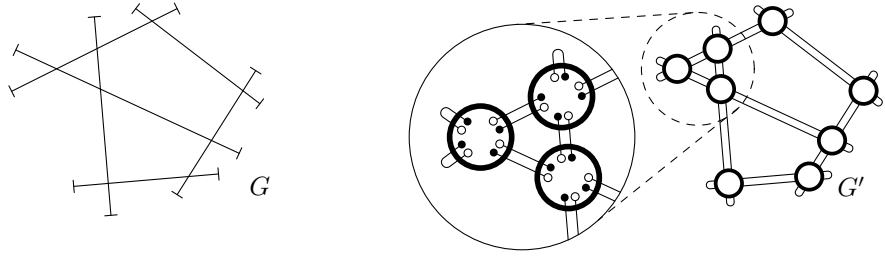}
\caption{We construct a planar graph $G'$ having the same topology as the arrangement of the segments.}
\label{reduction_sketch}
\end{figure}

We replace every crossing by a \emph{crossing gadget}. Every crossing gadget has four pairs of outer vertices. These vertices are incident with the
edges representing subsegments; see the detail in Figure~\ref{reduction_sketch}. In other words, two crossing gadgets are connected by a pair of parallel
edges if the crossings they represent lie on the same segment and there is no other crossing between them. A last subsegment of a segment is
represented by one edge connecting both outer vertices of a crossing gadget. The obtained planar graph has the same topology as the arrangement of the
segments.

\begin{figure}[b!]
\centering
\includegraphics{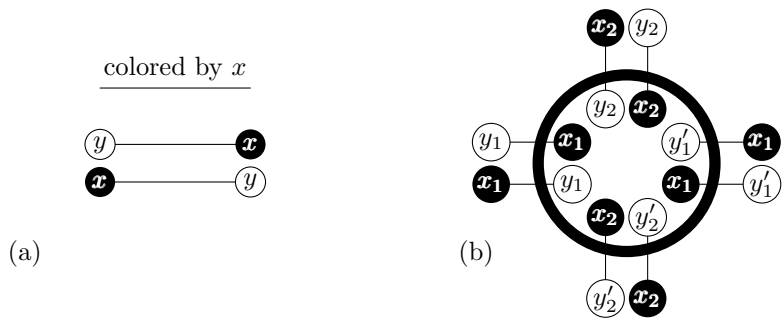}
\caption{(a) A subsegment colored by $x$ is represented by two parallel edges, with two \emph{color vertices} mapped to $x$ (depicted in black) and the other two
vertices mapped to an arbitrary $y$ (depicted in white) where $y \ne x$. (b) A crossing gadget transfers color information between the opposite
subsegments.}
\label{parallel_edges}
\end{figure}

\heading{Relation to Coloring.} Every subsegment is represented by a pair of parallel edges. Mapping of vertices of these edges to $K_6$ gives a
coloring of this subsegment in a way depicted in Figure~\ref{parallel_edges}a. On the other hand, crossing gadgets ensure that every covering of $K_6$
satisfy these properties. The vertices depicted in black are called \emph{color vertices} and the vertices depicted in white are called
\emph{non-color vertices}.

\heading{Crossing Gadget.} Every crossing gadget has four adjacent subsegments. Their color and non-color vertices alternate; see
Figure~\ref{parallel_edges}b. Using the topology of the segment arrangement, the crossing gadgets can be connected in this way so that $G'$ is planar.

This gadget ensures three properties:
\begin{packed_enum}
\item \emph{The subsegments are mapped in the way described in Figure~\ref{parallel_edges}a.}
\item \emph{The subsegments belonging to the same segment are colored by the same color.} Therefore, for every segment, all its subsegments are colored by
the same color and the color of this segment is well-defined. The crossing gadget gives no additional restrictions on non-color vertices.
\item \emph{Every two intersecting segments are colored by different colors.} It is possible to map the crossing gadget only if $x_1 \ne x_2$.
\end{packed_enum}

The crossing gadget is built from several basic blocks, called \emph{auxiliary gadgets}. The auxiliary gadget is a graph shown in
Figure~\ref{auxiliary_gadget}.

\begin{lemma} \label{auxiliary_gadget_lemma}
The auxiliary gadget can be mapped to $K_6$ in a unique way up to a permutation of the vertices of $K_6$.
\end{lemma}

\begin{proof}
Observe that if we fix a mapping for any vertex and its neighbors, the rest of the mapping is uniquely determined.\qed
\end{proof}

In every covering $f$, the six outer vertices $u_1, \dots, u_6$ of the auxiliary gadget are mapped to three distinct vertices of $K_6$ with $f(u_i) =
f(u_{i+3})$, $i \in \{1,2,3\}$. The parallel edges adjacent to the auxiliary gadget are mapped in the way described
in Figure~\ref{parallel_edges}a.

The crossing gadget consists of eight auxiliary gadgets; see Figure~\ref{crossing_gadget}. We need to prove that the crossing gadget is correct.

\begin{lemma} \label{crossing_gadget_lemma}
The crossing gadget can be mapped to $K_6$ if and only if the properties (1) to (3) are satisfied.
\end{lemma}

\begin{figure}[t!]
\centering
\includegraphics{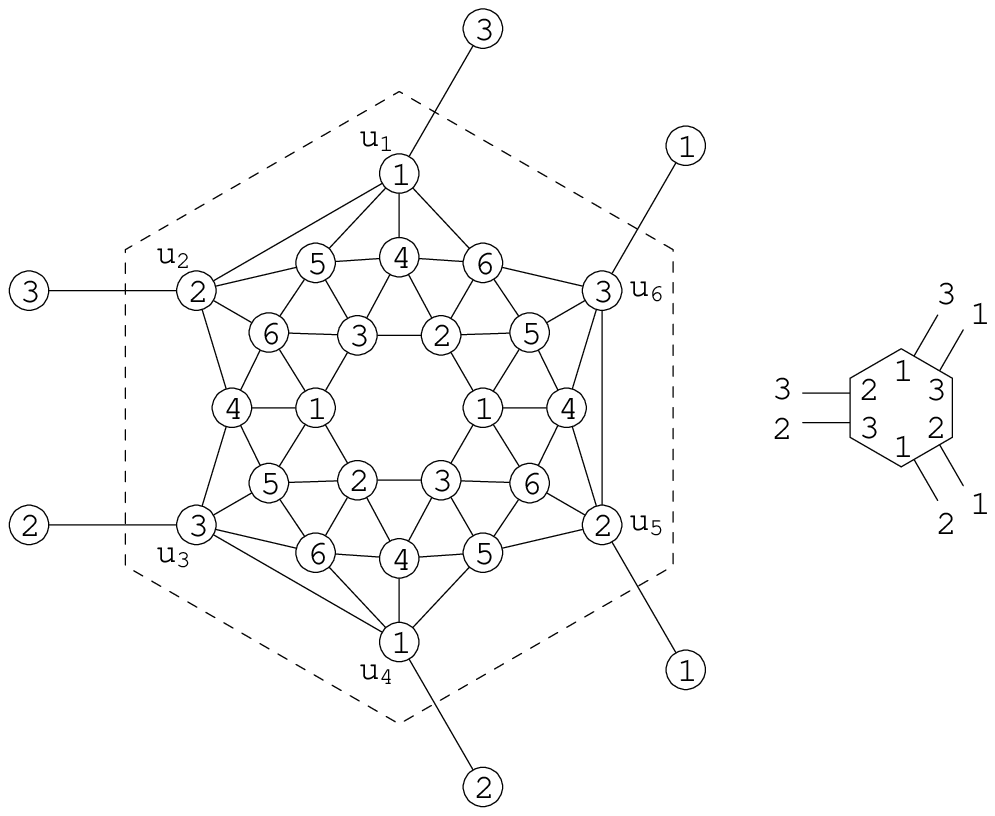}
\caption{The auxiliary gadget on the left, denoted by a hexagon on the right.}
\label{auxiliary_gadget}
\end{figure}

\begin{figure}[h!]
\centering
\includegraphics{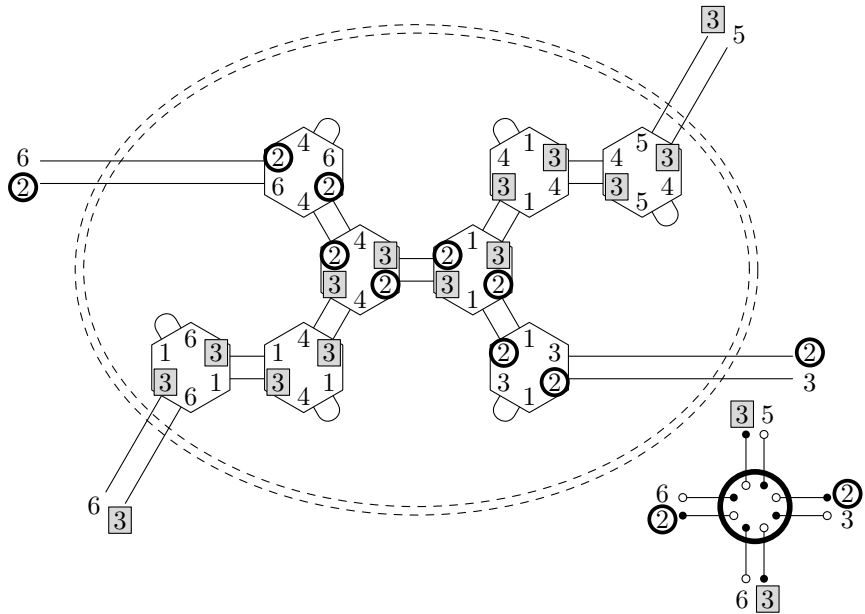}
\caption{The crossing gadget on the left, denoted by a circle on the right.}
\label{crossing_gadget}
\end{figure}

\begin{proof}
Let $G'$ cover $K_6$. Consider one crossing gadget. Since all edges representing subsegments are connected to auxiliary gadgets, according to
Lemma~\ref{auxiliary_gadget_lemma} these edges are mapped correctly as in Figure~\ref{parallel_edges}a. Colors are transfered between the opposite
subsegments, as depicted in Figure~\ref{crossing_gadget}. The central auxiliary gadgets force $x_1$ and $x_2$ to be distinct. Therefore, we know that
every mapping satisfies properties (1) to (3).

We need to show that if the vertices of the edges adjacent to the crossing gadget are mapped correctly, we can extend this mapping to the rest of the
gadget. By a straightforward case analysis we can see that an arbitrary correct mapping of non-color vertices can be extended.\qed
\end{proof}

We conclude this section with a proof of the main theorem:

\begin{proof}[Theorem~\ref{K6_is_NPc_theorem}]
Let $G'$ cover $K_6$. By Lemma~\ref{crossing_gadget_lemma}, the mapping of every crossing gadget satisfies the properties (1) to (3). Using the properties
(1) and (2), we can infer colors of the segments. By the property (3), this coloring is proper.

On the other hand, given a proper coloring, we map the color vertices according to this coloring. By Lemma~\ref{crossing_gadget_lemma}, it is possible to
extend the mapping to the entire graph $G'$.\qed
\end{proof}


\section{Hardness of Planar Covering of $K_4$, $K_5$, $K_4^+$ and $K_5^-$} \label{sec:hardness_others}

In this section, we sketch the proof of hardness of $\planarcover$ of $K_4$, $K_5$, $K_4^+$ and $K_5^-$.
The reductions slightly modify the reduction described in Section~\ref{sec:hardness_K6}.

In the case of $K_4$ and $K_5$, we just change the auxiliary gadget and reduce these problems from $4$-\segcol, resp. $5$-\segcol.

In the case of $K_4^+$ and $K_5^-$, we change both the auxiliary gadget and the crossing gadget and reduce these problems from $3$-\segcol. The color
vertices are mapped to $1$, $2$ and $3$, the non-color vertices are mapped to $0$ (or $-$ in the case of $K_5^-$); see
Figure~\ref{graphs_K4plus_and_K5minus}.

Due to space limitations, details of these reductions are described in Appendix.

\begin{figure}[t]
\centering
\includegraphics{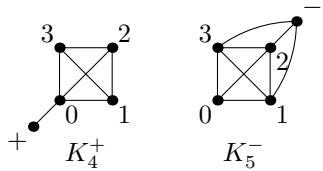}
\caption{Graphs $K_4^+$ and $K_5^-$ with labeled vertices.}
\label{graphs_K4plus_and_K5minus}
\end{figure}


\section{Hardness of Planar Covering of the Dumbbell Graph} \label{sec:hardness_dumbbell}

\begin{figure}[b]
\centering
\includegraphics{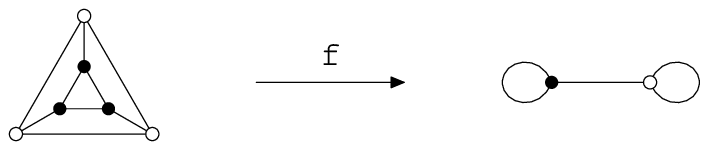}
\caption{An example of a cover of the dumbbell graph $W$.}
\label{multigraph_example}
\end{figure}

In this section, we prove hardness of $\planarcover(D)$ where $D$ is the dumbbell graph (see Figure~\ref{multigraph_example} on the right).  It is a
multigraph and the notion of covering can be extended to multigraphs, see~\cite{kratochvil94}. For the purpose of this paper, we need only the
following: $G$ is a planar cover of $D$ if $G$ is a cubic planar graph and can be colored by two colors (black and white) in such a way that every
black vertex has two black neighbors and one white neighbor and every white vertex has two white neighbors and one black neighbor; see
Figure~\ref{multigraph_example}.  In the rest of the section, we use this coloring interpretation.

%

To prove the hardness of $\planarcover(D)$, we first describe the problem we reduce from. $\mpsat$ is a satisfiability problem where:
\begin{packed_item}
\item all clauses contain exactly four variables,
\item the incidence graph of clauses and variables is planar, and
\item all variables are in the positive form, i.e. there is no negation.
\end{packed_item}
A clause is satisfied if exactly two variables are true. The entire formula is satisfied if all clauses are satisfied. For an example, see
Figure~\ref{sat_example}a. K\'ara, Kratochv\'il and Wood~\cite{kara_kratochvil_wood07} proved that this problem is still \cNP-complete.

\begin{figure}[t]
\centering
\includegraphics{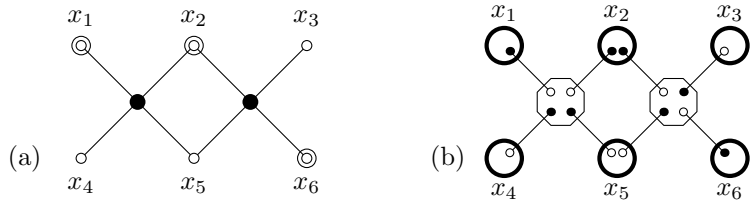}
\caption{(a) A graph $G$ representation of the formula: $(x_1,x_2,x_4,x_5) \wedge (x_2,x_3,x_5,x_6)$. This formula can be satisfied by an assignment
	$x_1 = x_2 = x_6 = 1$ and $x_3 = x_4 = x_5 = 0$. (b) The constructed graph $G'$ for this formula.} 
\label{sat_example}
\end{figure}

\heading{Overview of Reduction.} Let $G$ be a planar incidence graph of variables and clauses. We
construct a graph $G'$ such that $G'$ covers $D$ if and only if the formula is satisfiable. We
replace every variable with a \emph{variable gadget} and every clause with a \emph{clause gadget}.
If a variable is in a clause, we connect the variable gadget and the clause gadget by an edge. The
variable gadgets and the clause gadgets are connected in the way that the overall topology of $G$ is
preserved in $G'$; see Figure~\ref{sat_example}b.

The variable gadget can be colored in two ways which encodes the assignment of the variable. Every
clause gadget can be colored if and only if two of its variables gadget are true and the other two
are false.

\begin{figure}[b]
\centering
\includegraphics{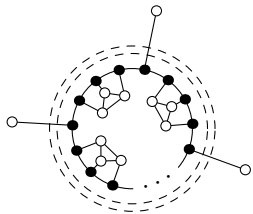}
\caption{The variable gadget with a unique coloring---up to swapping of the colors.}
\label{cinka_variable_gadget}
\end{figure}

\heading{Variable Gadget.} For a variable that appears in the formula $k$ times, the variable gadget contains the cycle $C_{4k}$.
Every fourth vertex of the cycle is connected to a clause gadget. The remaining vertices are connected to triangles; see
Figure~\ref{cinka_variable_gadget}.

\begin{figure}[t!]
\centering
\includegraphics{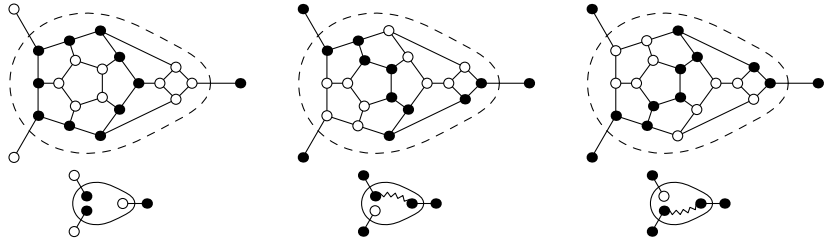}
\caption{The basic block has three different colorings---up to swapping of the colors.}
\label{cinka_block_gadget}
\end{figure}

\begin{figure}[t!]
\centering
\includegraphics{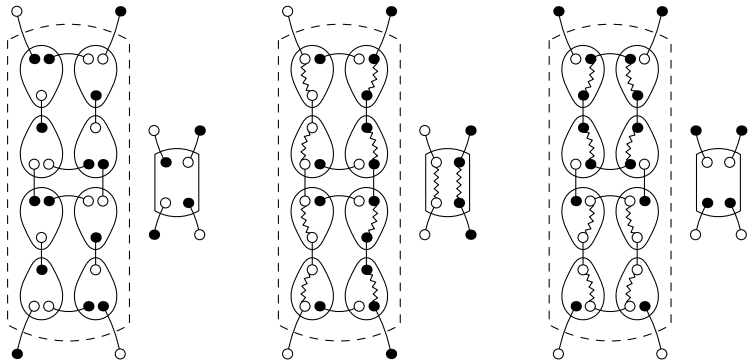}
\caption{The auxiliary gadget with all three colorings---up to swapping of the colors. Note that every coloring has two outer vertices black and the
other two white.}
\label{cinka_auxiliary_gadget}
\end{figure}

\begin{figure}[t!]
\centering
\includegraphics{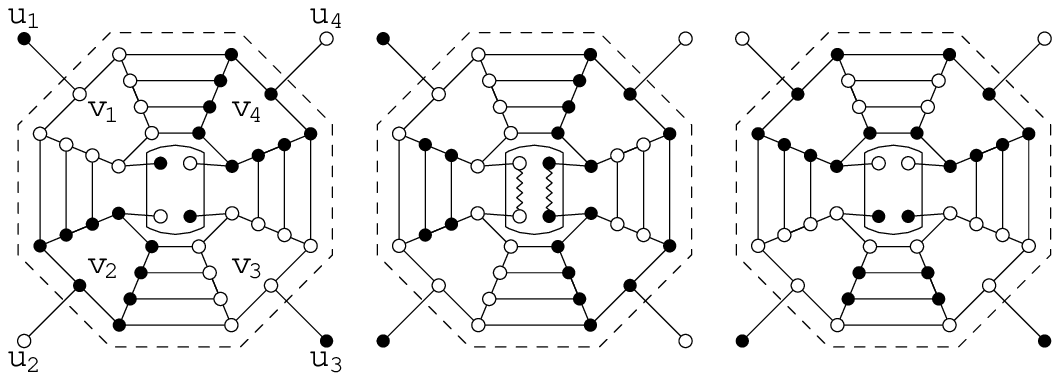}
\caption{The clause gadget with all three colorings---up to swapping of the colors.}
\label{cinka_clause_gadget}
\end{figure}
\begin{lemma} \label{cinka_variable_gadget_lemma}
For every coloring of the variable gadget, $u_1, \dots, u_k$ are colored by one color and $v_1, \dots, v_k$ by the other one.
\end{lemma}

\begin{proof}
Every triangle in the graph has to be monochromatic. The triangles force the cycle to be monochromatic as well.\qed
\end{proof}

The color of $v_1,\dots,v_k$ represents the value assigned to the variable. If the inner cycle is colored black, the variable is assigned true,
otherwise the variable is false.

\heading{Clause Gadget.} We start with \emph{basic blocks} described in Figure~\ref{cinka_block_gadget}. The \emph{auxiliary gadget} consists of eight
basic blocks; see Figure~\ref{cinka_auxiliary_gadget}. The reader is encouraged to prove that there exists no other colorings of these gadgets.

The clause gadget, described in Figure~\ref{cinka_clause_gadget}, contains an auxiliary gadget. Every clause gadget is connected by edges to the
corresponding variable gadgets.

\begin{lemma} \label{cinka_clause_gadget_lemma}
Let the vertices $u_i$ and $v_i$ have distinct colors for every $i \in \{1,2,3,4\}$. The crossing gadget can be covered if and only if exactly two of
$v_i$'s are colored black and the other two are colored white.
\end{lemma}

\begin{proof}
Observe that the coloring is forced by colors of $u_i$ and $v_i$. The rest is ensured by the auxiliary gadget; see description in
Figure~\ref{cinka_auxiliary_gadget}.\qed
\end{proof}

\begin{proof}[Theorem~\ref{cinka_is_NPc}]
We need to show that there exists a correct assignment of the variables if and only if $G'$ covers $D$. Let $G'$ cover $D$. According to
Lemma~\ref{cinka_variable_gadget_lemma}, every variable gadget has to be colored in one of two ways, one representing the true assignment and the other one
the false assignment; see Figure~\ref{cinka_variable_gadget}. By Lemma~\ref{cinka_clause_gadget_lemma}, the clause gadget can be colored if and only
if exactly two variables in the clause are true and the other two are false. Therefore, the colorings gives an assignment of the variables which
satisfies the formula.

On the other hand, if there exists a correct assignment, we cover the variable gadgets according to it. Since every clause has two variables assigned
true and the other two assigned false, the corresponding clause gadgets can be covered according to Lemma~\ref{cinka_clause_gadget_lemma}.
We obtain a correct colorings of $G'$.

The reduction is clearly polynomial, which concludes the proof.\qed
\end{proof}

\section{Conclusions}

\begin{figure}[b]
\centering
\includegraphics{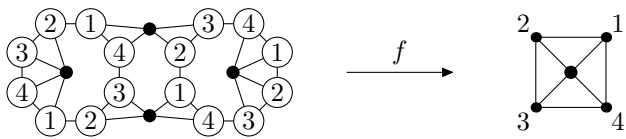}
\caption{An example planar cover of $W_4$. If $G$ covers $W_4$, it has to consist of cycles of
length divisible by four connected by black vertices of degree four. The cycles can be labeled
in the cyclic order $1$, $2$, $3$ and $4$ (eight possible labellings for each cycle) in such
a way that every vertex of degree four is adjacent to one vertex of each label.}
\label{example_of_wheel_cover}
\end{figure}

In this paper, we prove hardness of $\planarcover(H)$ for several small graphs $H$. Our techniques can be generalized to prove hardness of other
graphs. For example, if we replace the leaf in $K_4^+$ with any planar graph, the resulting $\planarcover$ problem is still \cNP-complete.

Our results give a positive answer to Question~\ref{kratochvil_question} for all graphs with at most
five vertices except for $W_4$, the wheel graph with four outer vertices ($\cover(W_4)$ is
\cNP-complete; see~\cite{kratochvil98}). For an example, see Figure~\ref{example_of_wheel_cover}.
Since the symmetries of $W_4$ are different from the symmetries of other graphs solved in this
paper, a reduction would require a new technique. We note that we were able to prove hardness of
partial $\planarcover(W_4)$.

\section*{Acknowledgment}
We would like to thank Jan Kratochv\'\i l for introducing us to the problem
and kindly answering our questions. We would also like to thank Aaron D.
Jaggard, Pavel Pat\' ak and Zuzana Safernov\'a for fruitful discussions.

\bibliographystyle{alpha}
\bibliography{planar_covering}

\newpage
\section*{Appendix}

In this appendix, we show hardness of $\planarcover$ of $K_4$, $K_5$, $K_4^+$ ($K_4$ with a leaf) and $K_5^-$ ($K_5$ without an edge).

\begin{figure}[b!]
\centering
\includegraphics{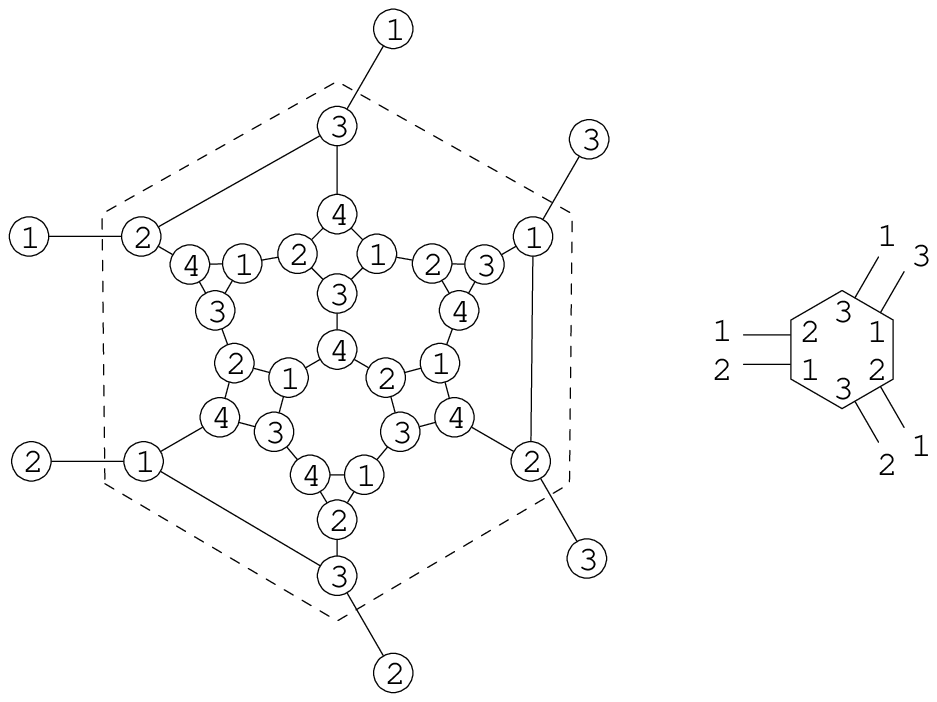}
\caption{The auxiliary gadget for $\planarcover(K_4)$.}
\label{K4_auxiliary_gadget}
\end{figure}

\begin{figure}[h!]
\centering
\includegraphics{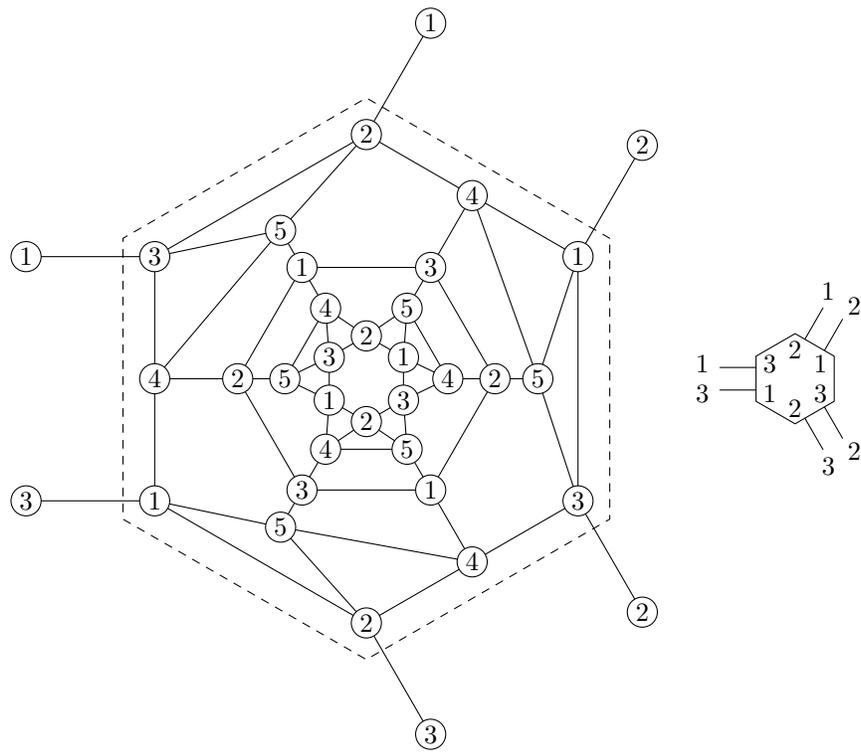}
\caption{The auxiliary gadget for $\planarcover(K_5)$.}
\label{K5_auxiliary_gadget}
\end{figure}

\heading{Covering of $K_4$ and $K_5$.} We modify the reduction described in Section~\ref{sec:hardness_K6}. Since these graphs have less vertices, we
reduce these problems from $4$-\segcol\ (resp.~$5$-\segcol). We use the auxiliary gadget from Figure~\ref{K4_auxiliary_gadget} (for $K_4$) and
Figure~\ref{K5_auxiliary_gadget} (for $K_5$).

To complete the reduction, we need to prove the following two properties. The constructed gadgets have the same properties as the auxiliary gadget for
$K_6$. They can be mapped to $K_4$ (resp.~$K_5$) in a unique way up to a permutation of the colors. Also, variations of
Lemma~\ref{crossing_gadget_lemma} holds. Both can be proved by a straightforward case analysis.

\heading{Covering of $K_4^+$ and $K_5^-$.} We reduce these problems from $3$-\segcol. We make more significant changes in the reduction. All color
vertices are mapped to $1$, $2$ or $3$. For $K_4^+$, all non-color vertices are mapped to $0$. For $K_5^-$, either all non-color are mapped to
$0$, or all of them are mapped $-$. We also modify the auxiliary gadget and the crossing gadget.

The auxiliary gadgets are described in Figure~\ref{K4plus_K5minus_auxiliary_gadget}. Using auxiliary gadgets, we construct the crossing gadget
described in Figure~\ref{K4plus_K5minus_crossing_gadget}. Lemma~\ref{crossing_gadget_lemma} holds for this crossing gadget and it can be proved by a
straightforward case analysis using properties of the auxiliary gadget. This concludes the reduction and proves the hardness.

\begin{figure}
\centering
\includegraphics{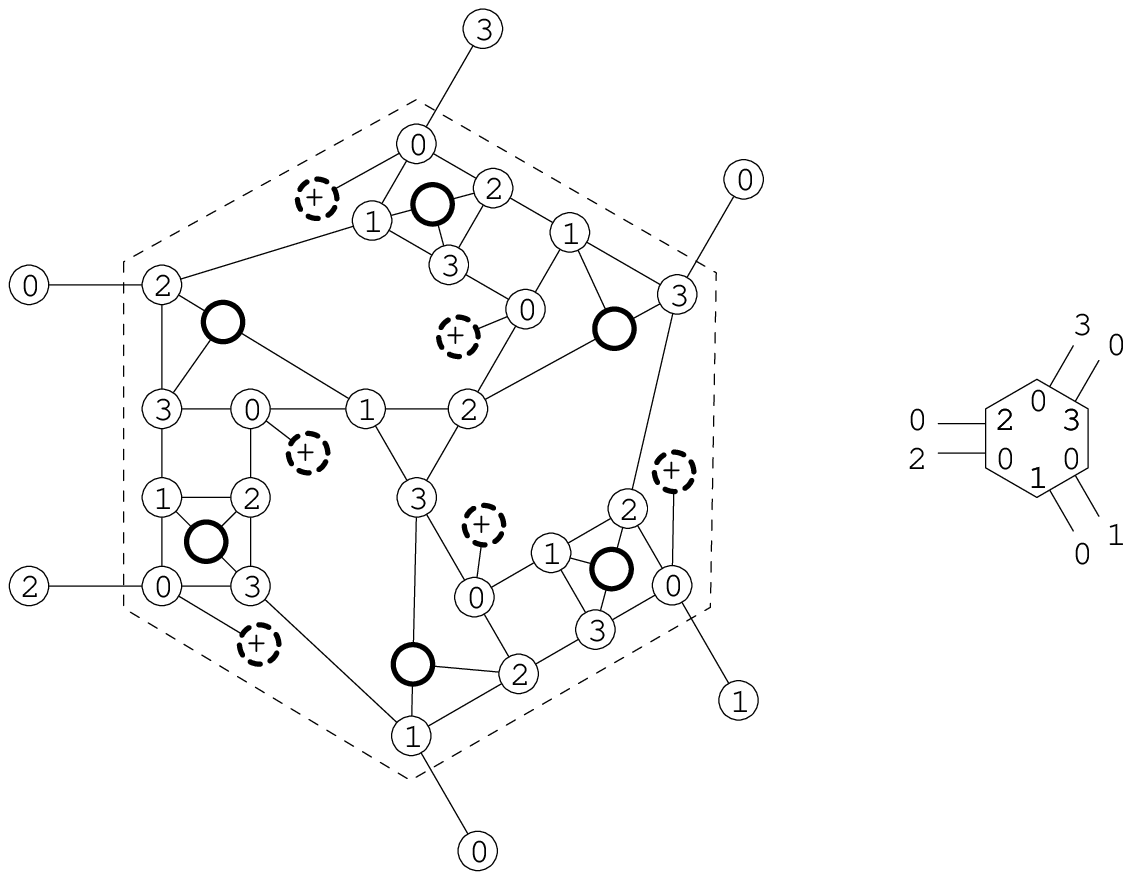}
\caption{The auxiliary gadgets for $\planarcover(K_4^+)$ and $\planarcover(K_5^-)$ in one figure. For $K_4^+$, the gadget does not contain vertices
mapped to $-$. For $K_5^-$, the gadget does not contain vertices mapped to $+$. Both gadgets admit only one mapping up to a permutation of $1$,
$2$ and $3$ (and in the case of $K_5^-$ up to swapping of $0$ and $-$).}
\label{K4plus_K5minus_auxiliary_gadget}
\end{figure}

\begin{figure}
\centering
\includegraphics{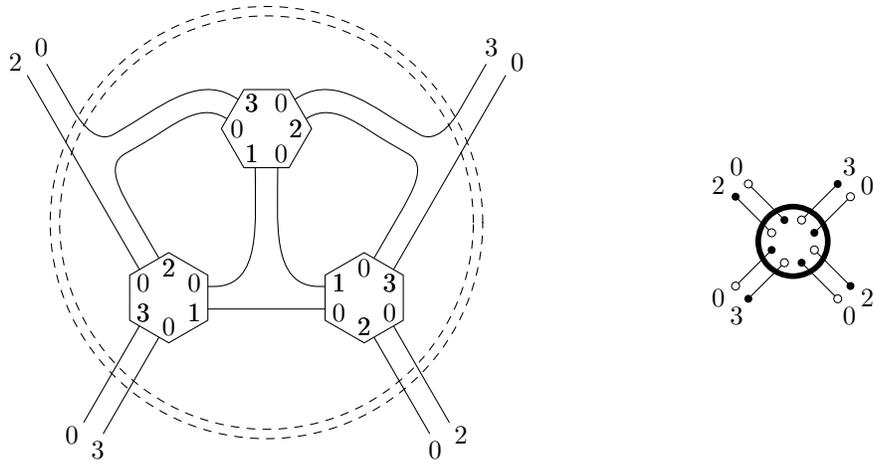}
\caption{The crossing gadget for \planarcover\ of $K_4^+$ and $K_5^-$.}
\label{K4plus_K5minus_crossing_gadget}
\end{figure}

\end{document}